\documentclass[12pt]{article}
\usepackage{graphicx}
\usepackage{amsmath, amssymb, amsthm}
\usepackage{mathrsfs}
\usepackage{geometry}
\usepackage{cite}
\usepackage[colorlinks=true]{hyperref}
\usepackage{authblk}

\title{Classify all representation which contains a Steinberg in its hyperspecial subgroup}
\author{Runze Wang}

\date{Peking University, Beijing, China; e-mail:runze\_wang@stu.pku.edu.cn}

\usepackage{setspace}
\newtheorem{theorem}{Theorem}[section]
\newtheorem{lemma}[theorem]{Lemma}
\newtheorem{corollary}[theorem]{Corollary}
\newtheorem{definition}[theorem]{Definition}

\begin{document}

\maketitle

\begin{abstract}

This paper addresses Question 1 posed by Dipendra Prasad in his recent problem list~\cite{Prasad2025}: classify all irreducible smooth representations $\pi$ of an unramified reductive $p$-adic group $G$ (over a characteristic zero local field) such that the space of $K_+$-fixed vectors $V^{K_+}$ (viewed as a representation of the finite reductive group $K/K_+$) contains the Steinberg representation. 
Prasad initially expected that such representations should be twists of the Steinberg representation or generic unramified representations. In a subsequent private communication, he refined the expectation: they should be precisely the unique generic irreducible subquotients of unramified principal series. 
We prove that any such $\pi$ must be Iwahori-spherical, hence a subquotient of some unramified principal series $I(\chi)=\operatorname{Ind}_B^G\chi$. 
By a detailed analysis of the Iwahori-Hecke algebra action on the Iwahori-fixed space $V^I$, we show that each principal series $I(\chi)$ contains \emph{exactly one} irreducible subquotient whose $K_+$-fixed vectors contain the Steinberg representation, and we give an explicit construction of this subquotient. 
Consequently, we obtain a bijection between $W$-orbits of unramified characters and irreducible representations containing the Steinberg representation in their hyperspecial subgroup. 
The proof is self-contained and uses only elementary Hecke algebra methods. In particular, the reduction from the $K$-action on $K_+$-fixed vectors to modules over the finite Hecke algebra $H(K,I)$ is spelled out in detail, a step that has not been made explicit in the literature.
\end{abstract}

\section{Introduction}
Let $F$ be a non-archimedean local field of characteristic zero, $\mathcal{O}_F$ its ring of integers, $\mathbb{F}_q$ its residue field, and let $G$ be a connected reductive group defined over $F$ which is \textbf{unramified}, i.e., quasi-split and split over an unramified extension of $F$. 
Denote by $K = G(\mathcal{O}_F)$ a hyperspecial maximal compact subgroup of $G(F)$ and by $K_+$ its pro-unipotent radical; then $K/K_+ \simeq G(\mathbb{F}_q)$ is a finite reductive group. Through inflation, representations of $G(\mathbb{F}_q)$ may be viewed as representations of $K$.

In a recent discussion meeting at IISER Bhopal (July 2025), Dipendra Prasad proposed a list of accessible open problems in representation theory~\cite{Prasad2025}. The first problem reads:

\begin{quote}
Classify all irreducible representations of $G(F)$ containing the Steinberg representation of $G(\mathbb{F}_q)$ as a $K$-subrepresentation (i.e., such that $V^{K_+}$ contains the Steinberg representation of $K/K_+$ as a constituent).
\end{quote}

Prasad initially suggested that such representations should fall into two families: (1) twists of the Steinberg representation of $G(F)$; (2) generic unramified representations (i.e., irreducible spherical principal series). 
In a subsequent private communication, he refined this expectation: the representations in question should be precisely the \textbf{unique generic irreducible subquotients} of unramified principal series.

The present paper confirms this refined expectation and gives a complete classification. Our main results are as follows.

\begin{itemize}
    \item \textbf{Existence and uniqueness.} 
    For every unramified principal series $I(\chi) = \operatorname{Ind}_B^G\chi$ (where $B$ is a Borel subgroup and $\chi$ an unramified character of a maximal torus), there exists \emph{exactly one} irreducible subquotient whose $K_+$-fixed vectors contain the Steinberg representation. 
    We give an explicit construction of this subquotient via the Iwahori-Hecke algebra.
    \item \textbf{Parametrization.} 
    The construction yields a bijection between $W$-orbits of unramified characters (where $W$ is the relative Weyl group) and irreducible representations of $G(F)$ whose $K_+$-fixed vectors contain the Steinberg representation.
\end{itemize}

Our proof is self-contained and uses only elementary methods from the theory of Iwahori-Hecke algebras. The main steps are:
\begin{enumerate}
    \item Show that any representation satisfying Prasad's condition must be Iwahori-spherical (hence a subquotient of some $I(\chi)$) using the Moy–Prasad theory of depth-zero representations and the Borel–Casselman theorem.
    \item Reduce the problem to an eigenvalue problem on the Iwahori-fixed space $V^I$ by comparing two actions on $V^I$: one coming from the finite Hecke algebra $H(K,I)$ of the $p$-adic group, and the other from the Iwahori–Hecke algebra of the finite Lie group $K/K_+$. We prove that these two actions coincide (Lemma~6.2), a step that has not been spelled out clearly in the literature.
    \item Solve the eigenvalue problem using Casselman's basis $\{\phi_{w,\chi}\}$ of $I(\chi)^I$ and Lansky's formulas for the action of affine reflections. The unique (up to scalar) solution is given explicitly by
    \[
    v_0 = \sum_{w\in W} (-1)^{\ell(w)} [IwI:I]^{-1} \phi_{w,\chi}.
    \]
    \item Show that the subrepresentation generated by $v_0$ contains a unique irreducible subquotient with the desired property. And we also show that this subquotient is the unique generic subquotient in a unramified principal series.
\end{enumerate}

\subsection*{Relation to previous work}
After the first version of this paper appeared on arXiv, we were informed by Manish Mishra and Maarten Solleveld that related results exist in the literature. 
We are very grateful to them for their careful reading and for pointing out the connections.

\begin{itemize}
    \item Manish Mishra~\cite{Mishra2016} has studied generic representations in $L$-packets. 
    In Theorem~12 of~\cite{Mishra2016}, he proved the existence and uniqueness theorem under the framework of unramified L-packets.
    In the split case, Theorem~3 of his joint work with Basudev Pattanayak~\cite{MishraPattanayak} gives an explicit Hecke-algebra model for the unique generic constituent of an unramified principal series.
   These results, which are part of a more general framework for $L$ packages, imply our classification in a broader context.
    \item Maarten Solleveld~\cite{Solleveld2012} developed a comprehensive theory of principal series representations for quasi-split reductive $p$-adic groups. 
    In Theorem~3.4 of~\cite{Solleveld2012}, he established that the condition used by Prasad is equivalent to the representation being generic.(To see that one has to know as well the reduction steps from the K-action on K+-invariant vector to modules over H(K,I), which is contained in this article. )
    His work, which is much more general, provides a conceptual explanation for our results.
\end{itemize}

Professor Solleveld also kindly pointed out that the reduction steps from the $K$-action on $K_+$-fixed vectors to modules over $H(K,I)$—which we work out in detail—are not explicitly written elsewhere, and that this part of our work can be regarded as an original contribution. 
We are deeply indebted to them for their generosity and for helping us locate our work within the larger framework.

\subsection*{Original contributions of this paper}
While the existence and uniqueness of generic subquotients of unramified principal series are known from general theory, our work provides the following new elements:
\begin{enumerate}
    \item \textbf{Explicit construction.} We give an explicit formula for a basis vector $v_0$ in the Iwahori-fixed space that generates the desired subquotient. This explicit description is new and may be useful for further computations .
    \item \textbf{Self-contained elementary proof.} Our proof uses only Hecke algebra calculations and does not rely on the more advanced machinery of $L$-packets or the local Langlands correspondence.
    \item \textbf{Detailed reduction steps.} The comparison between the finite Hecke algebra $H(K,I)$ and the Iwahori–Hecke algebra of the finite Lie group $K/K_+$ (Lemma~6.2) is made completely explicit. As noted by Solleveld, this step, has not been written out clearly in the literature.

\end{enumerate}

\subsection*{Organization of the paper}
Section~2 recalls the Moy–Prasad filtration and the theory of unrefined minimal $K$-types. 
Section~3 analyzes the cuspidal support of the irreducible components of $V^{K_+}$ and proves that any representation containing the Steinberg representation must be Iwahori-spherical. 
Section~4 collects necessary facts about unramified principal series, including Casselman's basis and Lansky's formulas for the action of affine reflections. 
Section~5 reviews the representation theory of finite Lie groups, focusing on the principal series and the associated Iwahori–Hecke algebra. 
Section~6 establishes an isomorphism between the finite Hecke algebra of the finite group and the finite Hecke algebra $H(K,I)$ of the $p$-adic group, and proves that their actions on $V^I$ coincide. 
Section~7 contains the core lemma and the proof of the main theorem, including the explicit construction of the subquotient. 
Finally, we include a corollary giving a bijection between $W$-orbits of unramified characters and the classified representations.
Section 8 provides  readers with another view of this result.

\section{Unrefined minimal $K$-types and the Moy–Prasad filtration}

In this section we recall the notion of Moy–Prasad filtrations and unrefined minimal $K$-types for a reductive $p$-adic group, following \cite{MoyPrasad1994, MoyPrasad1996}. Let $k$ be a non-archimedean local field of characteristic zero, and let $G$ be a connected reductive group defined over $k$. Let $\mathfrak{g}$ be the Lie algebra of $G$ and $\mathfrak{g}^{*}$ its dual.

\subsection{Moy–Prasad filtrations}

For a point $x$ in the Bruhat–Tits building of $G/k$, let $\mathcal{P}_x = \mathcal{P}_x \cap G(k)$ be the corresponding parahoric subgroup of $G(k)$. Moy and Prasad associate to $x$ a decreasing filtration $\{\mathcal{P}_{x,r}\}_{r\ge 0}$ of $\mathcal{P}_x$ by open compact subgroups, defined as follows: for a non-negative real number $r$, let $\mathcal{P}_{x,r}$ be the subgroup generated by certain congruence subgroups of the maximal torus and by the root subgroups $U_\psi$ for all affine roots $\psi$ satisfying $\psi(x) \ge r$; see \cite[2.6]{MoyPrasad1994} for the precise definition. One sets $\mathcal{P}_{x,0} = \mathcal{P}_x$ and $\mathcal{P}_{x,r+} = \bigcup_{s>r} \mathcal{P}_{x,s}$. For $r>0$ the quotients $\mathcal{P}_{x,r}/\mathcal{P}_{x,r+}$ are abelian.

There is an associated filtration of the Lie algebra $\mathfrak{g}$. Let $\mathfrak{g}_{x,r}$ be the $\mathfrak{o}$-lattice in $\mathfrak{g}$ obtained from the smooth $\mathfrak{o}$-group scheme corresponding to $\mathcal{P}_x$; its definition is given in \cite[3.2]{MoyPrasad1996}. Similarly, one defines a filtration $\{\mathfrak{g}_{x,r}^*\}_{r\in\mathbb{R}}$ of the dual $\mathfrak{g}^*$ by
$$
\mathfrak{g}_{x,r}^* = \{ X\in \mathfrak{g}^* \mid X(Y)\in \mathfrak{o} \text{ for all } Y\in \mathfrak{g}_{x,-r+} \},
$$
see \cite[3.5]{MoyPrasad1994}. The pairing
$$
\mathfrak{g}_{x,-r}^*/\mathfrak{g}_{x,(-r)+}^* \times \mathfrak{g}_{x,r}/\mathfrak{g}_{x,r+} \to \mathfrak{f}
$$
is non-degenerate and $\mathbf{M}_x(\mathfrak{f})$-equivariant, where $\mathbf{M}_x$ is the reductive quotient of the special fibre of the integral model attached to $x$; cf. \cite[3.7]{MoyPrasad1996}.

For a character $\chi$ of $\mathcal{P}_{x,r}$ trivial on $\mathcal{P}_{x,r+}$ (with $r>0$), one can realize $\chi$ as a character of the abelian group $\mathcal{P}_{x,r}/\mathcal{P}_{x,r+}$. Via the isomorphism $\mathcal{P}_{x,r}/\mathcal{P}_{x,r+}\cong \mathfrak{g}_{x,r}/\mathfrak{g}_{x,r+}$ of \cite[3.8]{MoyPrasad1996} and the above duality, such a character corresponds to a coset
$$
X + \mathfrak{g}_{x,-r+}^* \in \mathfrak{g}_{x,-r}^*/\mathfrak{g}_{x,-r+}^*.
$$

\begin{definition}  \cite{MoyPrasad1994}
\label{def:nondeg}
Let $r>0$. A character $\chi$ of $\mathcal{P}_{x,r}$ trivial on $\mathcal{P}_{x,r+}$ is called \emph{non-degenerate} if the corresponding coset $X + \mathfrak{g}_{x,-r+}^*$ contains no nilpotent elements.
\end{definition}

\subsection{Unrefined minimal $K$-types}

\begin{definition}[\cite{MoyPrasad1994}, 5.1]
An \emph{unrefined minimal $K$-type} is a pair $(\mathcal{P}_{x,r},\chi)$ where $x$ is a point of the Bruhat–Tits building of $G/k$, $r$ is a non-negative real number, and $\chi$ is a representation of $\mathcal{P}_{x,r}$ trivial on $\mathcal{P}_{x,r+}$ such that:
\begin{itemize}
    \item[(i)] if $r = 0$, then $\chi$ is a cuspidal representation of $\mathcal{P}_{x,0}/\mathcal{P}_{x,0+}$, inflated to $\mathcal{P}_x$;
    \item[(ii)] if $r > 0$, then $\mathcal{P}_{x,r} \neq \mathcal{P}_{x,r+}$ and $\chi$ is a non-degenerate character of the abelian group $\mathcal{P}_{x,r}/\mathcal{P}_{x,r+}$.
\end{itemize}
The number $r$ is called the \emph{depth} of the minimal $K$-type.
\end{definition}

\begin{definition}[\cite{MoyPrasad1996}]
\label{def:associate}
Two unrefined minimal $K$-types $(\mathcal{P}_{x,r},\chi)$ and $(\mathcal{P}_{y,s},\xi)$ are said to be \emph{associates} if they have the same depth $r = s$ and the following conditions hold:
\begin{itemize}
    \item[(i)] Zero depth case: $r = s = 0$. Then there exists an element $g\in G(k)$ such that $\mathcal{P}_x \cap g\mathcal{P}_y g^{-1}$ surjects onto both $\mathbf{M}_x(\mathfrak{f})$ and $\mathbf{M}_{gy}(\mathfrak{f})$, and the representation $\chi$ is isomorphic to $\operatorname{Ad}(g)\xi$.
    \item[(ii)] Positive depth case: $r = s > 0$. Let $X + \mathfrak{g}_{x,-r+}^*$ be the coset representing $\chi$ and $Y + \mathfrak{g}_{y,-r+}^*$ the coset representing $\xi$. Then the $G(k)$-orbit of $X + \mathfrak{g}_{x,-r+}^*$ intersects $Y + \mathfrak{g}_{y,-r+}^*$.
\end{itemize}
\end{definition}

\subsection{The main theorem on existence and depth}

The following fundamental result establishes the existence of an unrefined minimal $K$-type in any irreducible admissible representation, the uniqueness of its depth, and the associativity of all minimal $K$-types occurring in the same representation.

\begin{theorem}[\cite{MoyPrasad1994}]
\label{thm:existence}
Assume that $k$ is of characteristic zero. Let $(\pi,V_\pi)$ be an irreducible admissible complex representation of $G(k)$. Then there exists a non-negative rational number $\rho(\pi)$ with the following properties:
\begin{enumerate}
    \item[(1)] For some point $x$ in the Bruhat–Tits building of $G/k$, the space $V_\pi^{\mathcal{P}_{x,\rho(\pi)+}}$ is non-zero, and $\rho(\pi)$ is the smallest number with this property.
    \item[(2)] For any point $y$ in the building, if $W = V_\pi^{\mathcal{P}_{y,\rho(\pi)+}} \neq \{0\}$, then:
    \begin{itemize}
        \item[(i)] if $\rho(\pi) = 0$, any irreducible $\mathcal{P}_{y,0}$-submodule of $W$ contains a minimal $K$-type of depth zero of a parahoric subgroup $\mathcal{Q} \subset \mathcal{P}_y$;
        \item[(ii)] if $\rho(\pi) > 0$, any irreducible $\mathcal{P}_{y,\rho(\pi)}$-submodule of $W$ is a minimal $K$-type.
    \end{itemize}
\end{enumerate}
Moreover, any two minimal $K$-types contained in $\pi$ are associates of each other.
\end{theorem}

The number $\rho(\pi)$ is called the \emph{depth} of the representation $\pi$. In particular, every depth-zero irreducible representation contains an unrefined minimal $K$-type.

\section{The cuspidal support of the irreducible components of $K/K_{+}$}

\begin{lemma}
\label{lem:cuspidal-support}
Let $G$ be an unramified reductive $p$-adic group, let $K$ be a hyperspecial subgroup of $G$, and let $K_{+}$ be its corresponding pro-unipotent radical. Let $(\pi,V)$ be an irreducible smooth depth-zero representation of $G$. Then we may regard $V^{K_{+}}$ naturally as a representation of $K$ and hence also as a representation of the finite reductive group $K/K_{+}$. Under these assumptions, then for an arbitrary irreducible $K/K_{+}$-subrepresentation $W$ of $V^{K_{+}}$, there exists a facet $c$ and a cuspidal representation $\rho$  of the finite Lie group $G_c/G_{c+}$, such that the vertex associated to $K$ is contained in the closure of $c$ and W contains $\rho$. (We can also prove that $[c,\rho]$ is a unrefined minimal K-type of V ). Moreover, every irreducible component of the $K/K_{+}$ representation is a subrepresentation of a Harish-Chandra series of the finite Lie group $K/K_{+}$.
\end{lemma}

\begin{proof}
Since $V$ has depth zero by hypothesis, the space $V^{K_{+}}$ is non-zero.

Let $W$ be an irreducible $K/K_{+}$-subrepresentation of $V^{K_{+}}$.

By Theorem \ref{thm:existence}(2) , $W$ contains a depth-zero minimal $K$-type associated to a parahoric subgroup $G_{c}\subseteq K$ (where $c$ is a facet in the Bruhat–Tits building $\mathcal{B}(G)$). Denote this $G_{c}/G_{c+}$-cuspidal representation by $\rho$. The subgroup $G_{c}/K_{+}$ is a parabolic subgroup of $K/K_{+}$ with Levi factor $G_{c}/G_{c+}$. For simplicity, write this Levi decomposition as $Q = LU$. Where $L$ is the Levi factor and $U$ is its unipotent radical.

There exists a $Q$-embedding of $\rho$ into $W$. Since $\rho$ is $U$-invariant, the image of this embedding lies in $W^{U}$. Note that
$$
W^{U} \subset (V^{K_{+}})^{G_{c+}/K_{+}} = V^{G_{c+}}.
$$
Because $\rho$ is cuspidal, the pair $[\rho,c]$ is an unrefined minimal $K$-type of $V$ by the definition of the unrefined minimal K-type . 

Moreover, we can see that every irreducible component of $V^{K_{+}}$ is a subrepresentation of the parabolic induction from a cuspidal representation of a Levi (this is Harish-Chandra induction since we can find the parabolic and we need not use Deligne-Lusztig induction).
\end{proof}

To simplify terminology, if the representation $(\pi,V)$ of $G$ satisfies that the $K/K_{+}$-representation $V^{K_{+}}$ contains a Steinberg representation, we say that $V$ \textit{contains Steinberg}.

\begin{lemma}
\label{lem:steinberg-iwahori}
Let $G$ be as in Lemma \ref{lem:cuspidal-support}, let $I$ be its Iwahori subgroup, and let $(\pi,V)$ be an irreducible representation of $G$. If $V$ contains Steinberg, then $V$ must have a non-zero $I$-fixed vector.
\end{lemma}

\begin{proof}
If $V$ contains Steinberg, then there exists an irreducible constituent of $V^{K_{+}}$ isomorphic to the Steinberg representation. The cuspidal support of the Steinberg representation is $(T,\mathbf{1})$, where $T$ is a torus and $\mathbf{1}$ denotes the trivial representation. 

By Lemma \ref{lem:cuspidal-support}, there exists a chamber $c$ whose closure contains the vertex associated to $K$ (Assume  the Iwahori subgroup associated to $c$ is $I$ ), such  that Steinberg contains this trivial representation, and the unrefined minimal $K$-type of $V$ is $(c,\mathbf{1})$. This means that the $I/I_{+}$-representation $V^{I_{+}}$ must contain a trivial representation; hence $V^{I} \neq 0$.
\end{proof}

A classical result of Borel and Casselman relates the existence of Iwahori fixed vectors to principal series representations.

\begin{theorem}[Borel–Casselman, \cite{Borel1976}]
\label{thm:borel-casselman}
Let $G$ be a connected reductive group over a non-archimedean local field $k$, and let $I \subset G(k)$ be an Iwahori subgroup. For an irreducible smooth representation $\pi$ of $G(k)$, the following are equivalent:
\begin{enumerate}
    \item $\pi^{I} \neq 0$;
    \item $\pi$ is isomorphic to a subquotient of an unramified principal series representation of $G(k)$.
\end{enumerate}
\end{theorem}

Combining Theorem \ref{thm:borel-casselman} and Lemma \ref{lem:steinberg-iwahori}, we obtain the following immediate corollary.

\begin{corollary}
\label{cor:steinberg-principal-series}
Assume that $G$ satisfies the hypotheses of Lemma \ref{lem:cuspidal-support}. Let $(\pi,V)$ be an irreducible representation of $G$. If $V$ contains Steinberg, then $\pi$ must be a subquotient of an unramified principal series.
\end{corollary}

\begin{lemma}
Suppose $G$ is an unramified reductive $p$-adic group, and $(\pi,V)$ is an Iwahori-spherical irreducible representation of $G$. And fix an Iwahori subgroup $I$ which is contained in K. Then every irreducible $K/K_{+}$-subrepresentation $W$ of $V^{K_{+}}$ lies in the principal Harish-Chandra series  $\operatorname{Ind}_{I/K_{+}}^{K/K_{+}}(1)$ (i.e., parabolic induction from the trivial representation of a quasi-split torus).
\end{lemma}
\begin{proof}
For a fixed Iwahori subgroup $I$ of $G$ , By \cite[Theorem 8.4.10(2)]{KalethaPrasad2023}, we know that $I/I_{+}$ is the minimal Levi of $K/K_{+}$. Since every finite Lie group is quasi-split, $I/I_{+}$ is a torus, and $I/K_{+}$ is the Borel containing this torus.

Since $V$ is Iwahori-spherical, $V^I \neq 0$; hence the unrefined minimal $K$-type of $V$ is (the chamber corresponding to $I$, trivial representation). Since two unrefined minimal K type is associated, every unrefined minimal K type of V has the form (chamber, trivial representation).

By the proof of Lemma \ref{lem:cuspidal-support}, we know that for every irreducible component $W$ of $V^{K_{+}}$ , there exists a chamber c (assume the Iwahori subgroup associated to c is I'), such that W  is a subrepresentation of $\operatorname{Ind}_{I'/K_{+}}^{K/K_{+}}(1)$. It is well known that every Borel subgroups of a finite Lie group are rational-conjugated. Hence  $\operatorname{Ind}_{I'/K_{+}}^{K/K_{+}}(1)$ is isomorphic to  $\operatorname{Ind}_{I/K_{+}}^{K/K_{+}}(1)$. Hence the lemma holds.
\end{proof}

\section{Some facts about the unramified principal series}

\subsection{Iwahori fixed subspace of an unramified principal series}

The material in this subsection is derived from the seminal work of Casselman \cite{Casselman1980}. Let $G$ be an unramified, quasi-split, connected reductive group defined over a local field $k$. Let $B$ be a Borel subgroup of $G$ with Levi subgroup $T$. We denote by $K$ a hyperspecial subgroup of $G$ and by $I$ the Iwahori subgroup contained in $K$. Let $\chi$ be an unramified character of $T$. We define $I(\chi)$ as the unnormalized induced representation $\operatorname{Ind}_B^G \chi$.

In this subsection, we exhibit a canonical basis for the Iwahori fixed subspace $I(\chi)^I$. Proposition 2.1 of Casselman's paper \cite{Casselman1980} states that the dimension of $I(\chi)^I$ equals the order of the relative Weyl group of $G$. We can construct these functions more concretely. To proceed, we first recall some notation.

Let $C_c^{\infty}(G)$ denote the space of compactly supported, locally constant, complex-valued functions on $G$. We define a projection from $C_c^{\infty}(G)$ onto $I(\chi)$. For a function $f \in C_c^{\infty}(G)$, we define its projection onto $I(\chi)$ by the formula
$$
\mathcal{P}_{\chi}(f)(g) = \int_{B} \chi^{-1}(b) f(bg) \, db,
$$
where $db$ is a left-invariant Haar measure on the Borel subgroup $B$ normalized so that $\operatorname{vol}(K \cap B) = 1$.

For any subset $X \subset G$, let $\operatorname{ch}_X$ denote its characteristic function. Let $W$ be the relative Weyl group of $G$. In \cite[Proposition 2.1]{Casselman1980}, Casselman showed that the functions $\mathcal{P}_{\chi}(\operatorname{ch}_{B w B})$, as $w$ runs over all elements of $W$, form a basis for the vector space $I(\chi)^I$. We denote $\mathcal{P}_{\chi}(\operatorname{ch}_{B w B})$ by $\phi_{w,\chi}$. We formally state this result as follows.

\begin{theorem}[\cite{Casselman1980}]
\label{thm:casselman-basis}
The functions $\phi_{w,\chi}$ for $w \in W$ form a basis of $I(\chi)^I$.
\end{theorem}

To provide a more explicit description, we can write down the values of these functions concretely. For any $b \in B$, $w \in W$, and $i \in I$, we have
$$
\phi_{w,\chi}(b w i) = \chi(b).
$$
Furthermore, the support of $\phi_{w,\chi}$ is contained in the double coset $B w I$.

It is well known that the $K$-fixed subspace $I(\chi)^K$ is one-dimensional. It is not hard to see that the sum $\sum_{w \in W} \phi_{w,\chi}$ is a non-zero vector in $I(\chi)^K$.

\subsection{Effect of affine reflections on the Iwahori-fixed vectors}

In this section we recall the results of Lansky \cite{Lansky2001} on the action of the affine Weyl group on the space of Iwahori-fixed vectors of an unramified principal series representation. We adapt his notation to match our conventions: let $I$ denote an Iwahori subgroup (denoted by $B$ in \cite{Lansky2001}), and let $B$ denote a minimal parabolic subgroup (denoted by $P$ in \cite{Lansky2001}). The opposite parabolic is denoted $B^{-}$. We use unnormalized induction, i.e., the induced representation $\operatorname{Ind}_B^G \chi$ without the modular factor $\delta^{1/2}$; consequently, the formulas in Corollary 3.2 of \cite{Lansky2001} simplify by omitting all occurrences of $\delta^{1/2}$. The maximal split torus is denoted by $S$ (instead of $A$ in \cite{Lansky2001}).

We briefly recall the notation used in \cite{Lansky2001} which will be needed in the sequel. Let $G$ be a quasi-split connected reductive $p$-adic group, and fix a maximal split torus $S$ contained in a minimal parabolic $B$. And assume that the root system of G relative to S is $\Phi'$ and suppose that $\Phi$ is the sub root system vanishing on the hyperspecial point $x_0$. $\Phi$ is reduced and has a bijection with $^{nd}\Phi'$. Let $M$ be the Levi component of $B$ (the centralizer of $S$), and let $N$ be the unipotent radical of $B$. Write $B = MN$. Let $B^{-} = MN^{-}$ be the opposite parabolic. From now on, we normalize the Haar measure on $G$ such that $I$ has volume 1.

Let $W = N_G(S)/M$ be the relative Weyl group. Let $\mathcal{N}$ be the normalizer of $S$ in $G$, and let $M_0$ be the maximal compact subgroup of $M$. Define the Iwahori–Weyl group $\widetilde{W} = \mathcal{N}/M_0$. There is a natural surjection $\nu: \widetilde{W} \to W$. The affine Weyl group $W_{\mathrm{aff}}$ is a normal subgroup of $\widetilde{W}$ generated by affine reflections.

Let $\Phi$ be the reduced root system obtained from the affine roots vanishing at a fixed special point $x_0$; it is identified with the set of non-divisible relative roots. Let $\Phi^{+}$ be the positive roots determined by $B$, and let $\Delta$ be the corresponding set of simple roots. For each irreducible component $\Phi_i$ of $\Phi$, let $\widetilde{\alpha}_i$ denote its highest root, and let $\widetilde{\Delta}$ be the set of these highest roots.

The affine root system $\Phi_{\mathrm{aff}}$ consists of the affine functions $\alpha + k$ ($\alpha \in \Phi$, $k \in \mathbb{Z}$) and $\widetilde{\alpha} - 1$ ($\widetilde{\alpha} \in \widetilde{\Delta}$). The set of simple affine roots is
$$
\Delta_{\mathrm{aff}} = \{\alpha \mid \alpha \in \Delta\} \cup \{\widetilde{\alpha} - 1 \mid \widetilde{\alpha} \in \widetilde{\Delta}\}.
$$
For each $\alpha \in \Delta_{\mathrm{aff}}$, let $w_{\alpha}$ be the corresponding affine reflection in $W_{\mathrm{aff}}$. The set $S_{\mathrm{aff}} = \{w_{\alpha} \mid \alpha \in \Delta_{\mathrm{aff}}\}$ generates $W_{\mathrm{aff}}$ as a Coxeter group.

For $\alpha \in \Phi$, define the translation $a_{\alpha} = w_{\alpha} w_{\alpha-1} \in W_{\mathrm{aff}}$. One has $a_{-\alpha} = a_{\alpha}^{-1}$.

Let $I$ be an Iwahori subgroup fixed by a chamber in the apartment, and let $K_0$ be the special maximal compact subgroup corresponding to the special point $x_0$. The group $I$ has an Iwahori decomposition with respect to $B$ and $M$:
$$
I = (I \cap N^{-}) (I \cap M) (I \cap N).
$$
Set $I_{\alpha} = I \cap w_{\alpha} I w_{\alpha}$ for $\alpha \in \Delta$, and $I_{\widetilde{\alpha}-1} = I \cap w_{\widetilde{\alpha}-1} I w_{\widetilde{\alpha}-1}$ for $\widetilde{\alpha} \in \widetilde{\Delta}$.

For each affine root $\psi$, let $U_{\psi}$ denote the corresponding root subgroup. Then $I$ is generated by $M_0$ and the subgroups $U_{\psi}$ with $\psi > 0$. In particular,
$$
I \cap N = \prod_{\alpha \in \Phi^{+}} U_{\alpha},\qquad I \cap N^{-} = \prod_{\alpha \in \Phi^{+}} U_{-\alpha+1},
$$
and $I \cap M = M_0$.

For any $w \in \widetilde{W}$, denote by $q(w)$ the index $[I w I : I]$. For $\alpha \in \Phi_{\mathrm{aff}}$, let $q_{\alpha} = [U_{\alpha-1} : U_{\alpha}]$. Then one has
$$
q(w_{\alpha}) = q_{\alpha+1} \ (\alpha \in \Delta),\qquad q(w_{\widetilde{\alpha}-1}) = q_{\widetilde{\alpha}+2} = q_{\widetilde{\alpha}} \ (\widetilde{\alpha} \in \widetilde{\Delta}).
$$

Let $\chi$ be an unramified character of $M$ (trivial on $M_0$), extended to a character of $B$ by making it trivial on $N$. Define the unramified principal series representation $I(\chi) = \operatorname{Ind}_B^G \chi$ (unnormalized induction). The space of Iwahori-fixed vectors $I(\chi)^I$ has a basis $\{\phi_{w,\chi} \mid w \in W\}$ described in \cite[Prop. 2.1]{Casselman1980}. Concretely, for $p \in B$, $w' \in W$, $b \in I$,
$$
\phi_{w,\chi}(p w' b) = 
\begin{cases}
\chi(p) & \text{if } w' = w,\\
0 & \text{otherwise}.
\end{cases}
$$

From the results of \cite[Theorem 3.1]{Lansky2001} one deduces the action of the characteristic functions of the double cosets $I s I$ ($s \in S_{\mathrm{aff}}$) in the Iwahori–Hecke algebra $\mathcal{H}(G,I)$. The following corollary records the formulas we need, adapted to unnormalized induction.

\begin{corollary}
\label{cor:lansky-hecke}
For $w \in W$, $\alpha \in \Delta$, and $\widetilde{\alpha} \in \widetilde{\Delta}$, we have in $\mathcal{H}(G,I)$ (cf. \cite{Lansky2001}[3.2], adapted to unnormalized induction):
\begin{align}
\operatorname{ch}_{I w_{\alpha} I} * \phi_{w,\chi} &=
\begin{cases}
q_{\alpha+1} \phi_{w w_{\alpha},\chi} + (q_{\alpha+1} - 1) \phi_{w,\chi} & \text{if } w\alpha \in \Phi^{-},\\[4pt]
\phi_{w w_{\alpha},\chi} & \text{if } w\alpha \in \Phi^{+},
\end{cases} \\[8pt]
\operatorname{ch}_{I w_{\widetilde{\alpha}-1} I} * \phi_{w,\chi} &=
\begin{cases}
\chi(a_{w\widetilde{\alpha}}) \phi_{w w_{\widetilde{\alpha}},\chi} & \text{if } w\widetilde{\alpha} \in \Phi^{-},\\[4pt]
q_{\widetilde{\alpha}} \chi(a_{w\widetilde{\alpha}}) \phi_{w w_{\widetilde{\alpha}},\chi} + (q_{\widetilde{\alpha}} - 1) \phi_{w,\chi} & \text{if } w\widetilde{\alpha} \in \Phi^{+}.
\end{cases}
\end{align}
\end{corollary}

\section{Representation theory of the finite Lie group}

In this section let us suppose that $G$ is a reductive Lie group defined over $\mathbb{F}_q$, $B$ is its Borel, $T$ is the maximal torus contained in $B$, and $U$ is its unipotent radical. $F$ is the geometric Frobenius endomorphism. The material of this section comes from Chapter 4 of the book of Geck and Jacon \cite{GeckJacon}.

\begin{definition}
The ring $H=\operatorname{Hom}_{\mathbb{C}[G^F]}(\operatorname{Ind}_{B^F}^{G^F}\mathbf{1},\operatorname{Ind}_{B^F}^{G^F}\mathbf{1})^{\mathrm{opp}}$ (In order to simplify the terminology, from now on, we will omit "opp". )is called the Iwahori–Hecke algebra of $G$.
\end{definition}
\cite{GeckJacon} tells us the relations between irreducible subrepresentations of the principal Harish-Chandra series and simple $H$-modules.

\begin{theorem}
The functor $\Theta: M \longmapsto \operatorname{Hom}_{G^F}(\operatorname{Ind}_{B^F}^{G^F}\mathbf{1},M)$ induces a bijection from irreducible subrepresentations of the principal Harish-Chandra series to the simple Hecke algebra modules up to isomorphism. The action of this $H$-module is $h*f=f \circ h$ for $h\in H$ and $f \in \operatorname{Hom}_{G^F}(\operatorname{Ind}_{B^F}^{G^F}\mathbf{1},M)$.
\end{theorem}

We can go one step further using Frobenius reciprocity:
$$
\operatorname{Hom}_{G^F}(\operatorname{Ind}_{B^F}^{G^F}\mathbf{1},M)=\operatorname{Hom}_{B^F}(\mathbf{1},M)=M^{B^F}.
$$
Hence if we have a finite dimensional representation $V$ of $G^F$, and we assume further that every irreducible component of $V$ is a subrepresentation of $\operatorname{Ind}_{B^F}^{G^F}\mathbf{1}$, then we can take $B^F$-invariant vectors of $V$ to make it into an $H$-module. Also note that ${Ind}_{B^F}^{G^F}\mathbf{1}=\mathbb{C}[G^F/B^F]$.

Then for an arbitrary element $n$ in the normalizer $N_{G^F}(T)$, we can define an element of $H$:
$$
T_n:\operatorname{Ind}_{B^F}^{G^F}\mathbf{1} \longrightarrow \operatorname{Ind}_{B^F}^{G^F}\mathbf{1}
$$
which sends $xB^F$ to the sum of all cosets $yB^F$ such that $x^{-1}y \in B^FnB^F$. We also define $q_w=[B^FwB^F:B^F]$ for an arbitrary element $w$ in the Weyl group. Then we have:

\begin{theorem}
$\{T_w\mid w\in W\}$ forms a basis of the Iwahori–Hecke algebra $H$. If we further assume that $S$ is the set of simple reflections, then $\{T_w\mid w \in W\}$ forms a basis of $H$, and $H$ has the following presentation:
$$
H = \left\langle T_s\;(s\in S) \;\middle|\; 
\underbrace{T_sT_tT_s\cdots}_{m_{st}\text{ factors}} = \underbrace{T_tT_sT_t\cdots}_{m_{st}\text{ factors}},\;
(T_s - q_s)(T_s +1) = 0 \right\rangle.
$$
\end{theorem}

Finally, if we have a finite dimensional representation $V=\bigoplus_{i}V_i$ where each $V_i$ is irreducible, and we assume further that every irreducible component of $V$ is a subrepresentation of $\operatorname{Ind}_{B^F}^{G^F}\mathbf{1}$, then we can take the $B^F$-fixed vectors of $V$. We obtain an $H$-module $V^{B^F}=\bigoplus_i V_i^{B^F}$. It is well known that the Steinberg representation is a 1-dimensional representation on which $T_{\alpha}$ (where $\alpha$ is a simple root) acts as $-1$. Hence by Theorem 5.1, we obtain the following corollary:

\begin{corollary}
The multiplicity of the Steinberg representation in $V$ equals $\dim_\mathbb{C}\{ v \in V \mid T_\alpha v=-v \text{ for 
\ every \  simple \ root\ } \alpha\}$.
\end{corollary}
We should also note that by definition, $q_{s_\alpha}=q_{\alpha+1}$ for a simple root $\alpha$, where $q_{\alpha+1}$ is the element defined in Section 4.2.

\section{Relations between finite Lie group and $p$-adic group}
Two Hecke algebras occur in the previous sections, so we will compare these two algebras in this section. To simplify terminology, let $H=\operatorname{End}_{K/K_{+}}(\operatorname{Ind}_{I/K_{+}}^{K/K_{+}}(1))$, and let $H(K,I)$ be the finite Hecke algebra of the $p$-adic group (i.e., $I$-bi-invariant functions which are supported on $K$).

\begin{lemma}
There is an isomorphism $\phi$ between $H$ and $H(K,I)$ satisfying $\phi(T_{s_\alpha})=\operatorname{ch}_{Is_{\alpha}I}$ for every simple root $\alpha$.
\end{lemma}
\begin{proof}
This follows directly from the presentations of the two algebras.
\end{proof}

\begin{lemma}
Suppose $G$ is an unramified reductive $p$-adic group with hyperspecial subgroup $K$ and Iwahori subgroup $I$. Let $(\pi,V)$ be an irreducible Iwahori-spherical representation of $G$. As a representation of the $p$-adic group, $V$ has an action coming from $H(K,I)$. As a representation of the finite Lie group $K/K_{+}$, according to Section 5, after taking $I/K_{+}$-fixed vectors, $V^I$ is also an $H$-module. Then the two actions coincide.
\end{lemma}
\begin{proof}
Here we emphasize that the measure on $G$ is normalized such that the volume of $I$ is 1. By Lemma 3.5, every irreducible component of $V^{K_{+}}$ is a subrepresentation of $\operatorname{Ind}_{I/K_{+}}^{K/K_{+}}(1)$.

In order to show that these two actions coincide, we just need to check the generators $T_{s_\alpha}$ and $\operatorname{ch}_{Is_\alpha I}$. More concretely, we need to show the following equality: for every element $v \in V^I$
$$
T_{s_\alpha} v = \operatorname{ch}_{Is_\alpha I} v.
$$
Assume that the representatives of $Is_{\alpha}I/I$ are $\{u \mid u \in \Lambda\}$.
Then
By the discussion of the action of $T_{s_\alpha}$ in Section 5, and after applying Frobenius reciprocity, we have:
$$
\operatorname{Hom}_{K/K_{+}}(\operatorname{Ind}_{I/K_{+}}^{K/K_{+}}(1),V^{K_{+}})=\operatorname{Hom}_{I/K_{+}}(1,V^{K_{+}})=V^I
$$
For an element $v \in V^I$, $v$ corresponding to the function $f_v$, where $f_v(gI/K_{+})=gv$, hence 
$$
T_{s_\alpha}v=f_v \circ T_{s_\alpha}(I/K_{+})
$$
$$
 = f_v\bigl(\sum_{u \in  \Lambda}uI/K_{+}\bigr).
$$
$$
= \sum_{u \in \Lambda} u v,
$$
$$
\text{RHS} = \int_{Is_{\alpha}I} x v \, dx = \sum_{u\in \Lambda} u v.
$$
Hence LHS = RHS.
\end{proof}

\section{The classification of irreducible representations containing Steinberg}

\subsection{Main Theorem}

\begin{theorem}
Suppose that $G$ is an unramified reductive $p$-adic group, $K$ is its hyperspecial subgroup, and $I$ is its Iwahori subgroup. If an irreducible representation $(\pi,V)$ of $G$ contains Steinberg, then it must be an Iwahori-spherical representation. Moreover, for every unramified principal series of $G$, there is exactly one subquotient that contains Steinberg, and we will give a concrete construction in the proof of the theorem.
\end{theorem}

\begin{proof}
The first statement follows from Corollary \ref{cor:steinberg-principal-series}.

We now prove the second statement. Suppose that $(\pi,V)$ is an Iwahori-spherical irreducible representation, and it is a subquotient of the unramified principal series $I(\chi)$.

$V^{K_{+}} \subset I(\chi)^{K_{+}}$ is a sub-$K/K_{+}$ representation. Since $V$ is a subquotient of $I(\chi)$, the space $V^{K_{+}}$ is a subquotient of $I(\chi)^{K_{+}}$. The latter is a principal series representation of the finite reductive group $K/K_{+}$; hence all its irreducible constituents lie in the principal Harish-Chandra series. Consequently, every irreducible component of $V^{K_{+}}$ lies in the principal series.

By the representation theory of finite groups of Lie type \cite{DigneMichel2020}, we can take the $I/K_{+}$ fixed points of $V^{K_{+}}$ (where $I/K_{+}$ is the Frobenius fixed point of a Borel subgroup of the finite Lie group $K/K_{+}$) to make it into a module of the Iwahori–Hecke algebra of the finite Lie group. So we need to study $V^I = (V^{K_{+}})^{I/K_{+}}$.

By Lemma 6.2, the two Hecke algebras have the same action; we can make $V^I$ an $H$-module of the finite Lie group.

In order to find the multiplicity of the Steinberg representation in $V^{K_{+}}$, we need to compute the dimension of the space
$$
W = \{ v \in V^I \mid \operatorname{ch}_{I s_\alpha I} v = -v \text{ for \ every \ simple \ root\  } \alpha \}.
$$

We use the basis mentioned in Section 4; suppose that $v = \sum_{w \in W} c(w) \phi_{w,\chi}$. We prove the following lemma:

\begin{lemma}
\label{lem:unique-eigenvector}
There is a unique (up to scalar) non-zero element of $I(\chi)$ that belongs to the space $W$.
\end{lemma}

\begin{proof}
For a fixed simple root $\alpha$, we can pair the elements of $W$ as $\{(w, w s_\alpha) \mid w\alpha \text{ is positive}\}$. According to Corollary \ref{cor:lansky-hecke}:
$$
\operatorname{ch}_{I s_\alpha I}(c(w)\phi_{w,\chi} + c(w s_\alpha)\phi_{w s_\alpha,\chi})
$$
$$
= (c(w) + c(w s_\alpha)q_{\alpha+1} - c(w s_\alpha))\phi_{w s_\alpha,\chi} + (c(w s_\alpha)q_{\alpha+1})\phi_{w,\chi}.
$$

Since $v$ lies in $W$, we see that
$$
-c(w) = c(w s_\alpha)q_{\alpha+1}
$$
and
$$
-c(w s_\alpha) = c(w) + c(w s_\alpha)q_{\alpha+1} - c(w s_\alpha).
$$

Hence we conclude that
$$
c(w) + q_{\alpha+1}c(w s_\alpha) = 0.
$$

Since $\alpha$ is arbitrary, we finish the uniqueness of the lemma via induction:

Let $c(1) = 1$. From the above, we obtain $c(s_\alpha) = -q^{-1}_{\alpha+1}$. Suppose we already know $c(w)$ for every $l(w) = m$. Now consider an element $u \in W$ with length $m+1$. There exists a simple root $\alpha$ such that $l(u s_\alpha)=l(u)-1$; then by the above argument $c(u)$ can be determined by $c(u s_\alpha)$, hence by the induction hypothesis we prove uniqueness.
In other words, $\dim W \le 1$.
Next, let us prove existence: set $c(w)=(-1)^{l(w)}([IwI:I])^{-1}$.
Note that the basic property of BN-pairs says:
If for $w\in W$ and a simple reflection $s_\alpha$ we have $l(w s_\alpha)=l(w)+1$, then $[IwI:I][Is_\alpha I:I]=[I w s_\alpha I:I]$.
This property tells us that there exists a non-zero element in $W$.
Since if $w \in W$ such that $w(\alpha)>0$, for a simple root $\alpha$. then $l(ws_\alpha)=l(w)+1$, using the above property we can see that $c(ws_\alpha)q_{\alpha+1}+c(w)=0$. Hence $\sum_{w \in W}(-1)^{l(w)}[IwI:I]^{-1}\phi_{w,\chi}$ is a non-zero element in W.
\end{proof}

Returning to the proof of our main theorem, we see that the dimension of $W$ is 1, and we conclude that there is only one subquotient of the series $I(\chi)$ which contains Steinberg.

Finally, we give a construction of this subquotient. Denote the element in $W$ by $v_0$. Let $V$ be the subrepresentation of $I(\chi)$ generated by $v_0$, and let $V_0$ be the maximal subrepresentation of $V$ which does not contain $v_0$. Let's us show that such $V_0$ exists:
Since $I(\chi)$ is of finite length, suppose that the Jordan-H\"{o}lder sequence of $I(\chi)$ is $$0\subset U_1 \subset U_2....\subset U_m=I(\chi)$$
And by the previous argument, we can assume that the irreducible subquotient $U_i/U_{i-1}$ contains a Steinberg. Then $v_0 \in U_i^I$, hence $V \subset U_i$, and since $v_0  \notin U_{i-1}^I$, that means $V \cap U_{i-1} \neq V$, and there exists a natural G-injection $V/(V\cap U_{i-1}) \xrightarrow{ }U_i/U_{i-1}$. Since $U_i/U_{i-1}$ is irreducible and $V/(V\cap U_{i-1}) \neq 0$, we can see that $V/(V\cap U_{i-1})$ is isomorphic to $U_{i}/U_{i-1}$. Hence we can take $V_0=V \cap U_{i-1}$, such $V_0$ exists. 
Then $V/V_0$ is the unique subquotient of $I(\chi)$ that contains the Steinberg representation.
\end{proof}
It is well known that two principal series $I(\chi)$ and $I(w\chi)$ have the same Jordan-H\"{o}lder factors (Where w is an element of the Weyl group). And if $\chi$ and $\chi'$ are not in the same $W$ -orbit, then $I(\chi)$ and $I(\chi')$ do not have a common Jordan-H\"{o}lder factor. Hence we get a corollary:
\begin{corollary}
    There is a bijection between $W$-orbits of unramified  characters and irreducible representations containing the Steinberg  in its hyperspecial subgroup.
\end{corollary}

Remark:   In Theorem~3.4 of~\cite{Solleveld2012}, Professor Solleveld established that the condition used by Prasad is equivalent to the representation being generic.(To see that one has to know as well the reduction steps from the K-action on K+-invariant vector to modules over H(K,I), which is contained in this article. )
\section{Another Perspective: generic representation}
\begin{theorem}
The representation constructed above is the unique generic subquotient in the unramified principal series $I(\chi)$.
\end{theorem}
\begin{proof}
It is well known that there is exactly one subquotient of $I(\chi )$ is generic. And the work of Moeglin and Waldspurger \cite{MoeglinWaldspurger1987} told us an irreducible representation $(\pi,V)$ is generic if and only if the Wavefront set of $\pi$ is the regular orbit.
From Example 2.3.8 of Okada's paper \cite{Okada2022} we can see that if $V^{K_+}$ contains Steinberg, then the geometry wavefront set of this representation must be the regular nilpotent orbit. Hence the representation that we construct above is the unique generic subquotient.
\end{proof}

\end{document}